\documentclass[12pt]{article}

\usepackage{graphicx,url}
\usepackage{amsmath,amsthm,amssymb,color}
\usepackage[noadjust,sort]{cite}

\usepackage[colorlinks=true,citecolor=black,linkcolor=black,urlcolor=blue]{hyperref}

\allowdisplaybreaks

\makeatletter
\g@addto@macro\normalsize{%
  \setlength\abovedisplayskip{8pt plus 3pt minus 3pt}
  \setlength\belowdisplayskip{8pt plus 3pt minus 3pt}
  \setlength\abovedisplayshortskip{6pt plus 3pt minus 2pt}
  \setlength\belowdisplayshortskip{6pt plus 3pt minus 2pt}
}
\makeatother

\date{}

\normalsize

\numberwithin{equation}{section}

\newcommand\calL{\mathcal{L}}
\renewcommand\dfrac[2]{\lower0.15ex\hbox{\large$\frac{#1}{#2}$}}
\newcommand\nicebreak{\vskip 0pt plus 50pt\penalty-300\vskip 0pt plus -50pt }
\renewcommand\dfrac[2]{\lower0.18ex\hbox{\Large$\frac{#1}{#2}$}}
\newcommand\eps{\varepsilon}
\newcommand\abs[1]{\mathopen|#1\mathclose|}
\renewcommand\({\bigl(}
\renewcommand\){\bigr)}

\newcommand\red[1]{#1}

\newtheorem{thm}{Theorem}[section]

\newtheorem{lemma}[thm]{Lemma}
\newtheorem{conj}[thm]{Conjecture}

\def\dfrac#1#2{\lower0.15ex\hbox{\large$\textstyle\frac{#1}{#2}$}}
\def\({\bigl(}
\def\){\bigr)}

\def\({\bigl(}
\def\){\bigr)}

\def\E{\mathbb{E}}

\def\nicebreak{\vskip 0pt plus 50pt\penalty-300\vskip 0pt plus -50pt }
\let\originalleft\left
\let\originalright\right
\renewcommand{\left}{\mathopen{}\mathclose\bgroup\originalleft}
\renewcommand{\right}{\aftergroup\egroup\originalright}

\begin{document}

\title{Factorisation of the complete bipartite graph into spanning semiregular factors}

\author{
Mahdieh Hasheminezhad\\
\small Department of Computer Science\\[-0.8ex]
\small Yazd University\\[-0.8ex]
\small Yazd, Iran\\
\small\tt hasheminezhad@yazd.ac.ir
\and
Brendan D. McKay\\
\small School of Computing\\[-0.8ex]
\small Australian National University\\[-0.8ex]
\small Canberra, ACT 2601, Australia\\
\small\tt brendan.mckay@anu.edu.au
}

\maketitle

\begin{abstract}
We enumerate factorisations of the complete bipartite graph into spanning
semi\-regular graphs in several cases, including when the degrees of all the factors
except  one or two are small.  The resulting asymptotic behaviour is 
seen to generalise the number of semiregular graphs in an elegant way.
This leads us to conjecture a general formula when the number of
factors is vanishing compared to the number of vertices.
As a corollary, we find the average number of ways to partition the edges
of a random semiregular bipartite graph into spanning semiregular subgraphs
in several cases.
Our proof of one case uses a switching argument to find the probability
that a set of sufficiently sparse semiregular bipartite graphs are
edge-disjoint when randomly labelled.
\end{abstract}

\nicebreak

\section{Introduction}

A classical problem in enumerative graph theory is the asymptotic
number of 0-1 matrices with uniform row and column sums;
equivalently, semiregular bipartite graphs.

We will consider all bipartite graphs to have bipartition $(V_1,V_2)$
where $\abs{V_1} = m$ and $\abs{V_2}=n$.
An \textit{$(m,n,\lambda)$-semiregular} bipartite graph has every degree
in $V_1$ equal to~$\lambda n$ and every degree in $V_2$ equal to~$\lambda m$.
Of course, for such a graph to exist we must have $0\le\lambda\le 1$
and $\lambda m, \lambda n$ must be integers.  We
will tacitly assume that these elementary conditions hold throughout
the paper for every mentioned semiregular graph.
The parameter $\lambda$ will be called the \textit{density}.
Define $R_\lambda(m,n)$
to be the number of $(m,n,\lambda)$-semiregular bipartite graphs.
The asymptotic determination of $R_\lambda(m,n)$ as $n\to\infty$
with $m=m(n)$ and $\lambda=\lambda(n)$ is not yet complete, but
the known values fit a simple formula.

\begin{thm}\label{thm:reg}
   Let $n\to\infty$ with $m\le n$.  Then
   \[
       R_\lambda(m,n) \sim
       \frac{\displaystyle\binom{n}{\lambda n}^{\!m} \binom{m}{\lambda m}^{\!n}}
              {\displaystyle\binom{mn}{\lambda mn}}
       \, (1-1/m)^{(m-1)/2}
  \]
  in the following cases.
  \begin{itemize}\itemsep=0pt
      \item[1.] $\lambda = o\((mn)^{-1/4}\)$.
      \item[2.] For sufficiently small $\eps>0$ ,
           $(1-2\lambda)^2\(1+\dfrac{5m}{6n}+\dfrac{5n}{6m}\)
                      \le(4-\eps)\lambda(1-\lambda)\log n$
           and $n=o\(\lambda(1-\lambda) m^{1+\eps}\)$.
       \item[3.] For some $\eps>0$,
           $2\le m = O\( (\lambda(1-\lambda)n)^{1/2-\eps}\)$.
       \item[4.] $\lambda\le c$ for a  sufficiently small constant~$c$, 
          $n=O(\lambda^{1/2-\eps}m^{3/2-\eps})$ for some $\eps >0$,
          and $\lambda m\ge \log^K n$ for every $K$. 
  \end{itemize}
\end{thm}
\begin{proof}
    Note that $(1-1/m)^{(m-1)/2}\to e^{-1/2}$ if $m\to\infty$.
   Case~1 covers the sparse case and was proved by McKay and Wang~\cite{MW}.
   Case~2 applies when $m$ and $n$ are not very much different and the
   graph is very dense, while Case~3 covers all densities when $n$ is much
   larger than~$m$.
   Case~2 was proved by Greenhill and McKay~\cite{GMX},
   and Case~3 by Canfield and McKay~\cite{CM}.
   Case~4, proved by Liebenau and Wormald~\cite{LW}, covers a wide range
   of densities and moderately large~$n/m$.
 \end{proof}
  
  \red{Theorem~\ref{thm:reg} does} not cover all $(m,n,\lambda)$.
   For example $\lambda=\frac12,m=n^{2/3}$ is missing,
   and so is $\lambda=\frac 1m, m=n^{1/2}$.
   However, based on numerical evidence a strong form of the
   theorem was conjectured in~\cite{CM} to hold for all $(m,n,\lambda)$.

We can consider $R_\lambda(m,n)$ \red{as counting} the partitions of the edges of~$K_{m,n}$
into two spanning semiregular subgraphs, one of density $\lambda$ and one of
density $1-\lambda$.
This suggests a generalization: how many ways are there to partition the edges
of~$K_{m,n}$ into more than two spanning semiregular subgraphs, of specified densities?

For positive numbers $\lambda_0,\ldots,\lambda_k$ with sum~1,
define $R(m,n; \lambda_0,\ldots,\lambda_k)$ to be the number of ways to partition
the edges of
$K_{m,n}$ into spanning semiregular subgraphs of density $\lambda_0,\ldots,\lambda_k$.
We conjecture that for $k=o(m)$, the asymptotic answer is a simple
generalisation of Theorem~\ref{thm:reg}.

\begin{conj}\label{conj:main}
  Let $\lambda_0,\ldots,\lambda_k$ be positive numbers such that
  $\sum_{i=0}^k\lambda_i=1$.
  Then, if $n\to\infty$ with $2\le m\le n$, and $1\le k=o(m)$,
  $R(m,n; \lambda_0,\ldots,\lambda_k)\sim R'(m,n; \lambda_0,\ldots,\lambda_k)$,
  where (using multinomial coefficients)
   \[
     R'(m,n; \lambda_0,\ldots,\lambda_k) =  
       \frac{\displaystyle\binom{n}{\lambda_0 n,\ldots,\lambda_k n}^{\!m}
                                   \binom{m}{\lambda_0 m,\ldots,\lambda_k m}^{\!n}}
              {\displaystyle\binom{mn}{\lambda_0 mn,\ldots,\lambda_k mn}}
       \, (1-1/m)^{k(m-1)/2}.
   \]
\end{conj}

We will prove the conjecture in six cases.

\begin{thm}\label{thm:main}
Conjecture~\ref{conj:main} holds in the following cases.
\begin{enumerate}\itemsep=0pt
  \item[(a)] $k=1$ and one of the conditions of Theorem~\ref{thm:reg} holds.
   \item[(b)] $k\ge 2$, $m\le n$ and $m^{-1}n^3 \(\sum_{i=1}^k\lambda_i\)^3 = o(1)$.
  \item[(c)] $k\ge 2$, $m \le n$ and $m^{1/2}n\sum_{1\le i<j\le k} \lambda_i\lambda_j=o(1)$.
   \item[(d)] $m=n$, $k=o(n^{6/7})$ and $\lambda_1=\cdots=\lambda_k=\frac 1n$
     (the case of Latin rectangles).
  \item[(e)] $(m,n,\lambda_1)$ satisfies condition 2 of Theorem~\ref{thm:reg},
     and $\lambda_2+\cdots+\lambda_k=O(n^{-1+\eps})$ for sufficiently
     small $\eps>0$.
  \item[(f)] $m=O(1)$ and $1\le k\le m-1$.  In this case we do not need the condition $k=o(m)$.
\end{enumerate}
\end{thm}

Part~(a) is just a restatement of Theorem~\ref{thm:reg}.
\red{Part~(b)} will be proved using~\cite{silver}.
Part~(c) will follow from a switching argument applied to the probability of two
randomly labelled semiregular graphs being edge-disjoint.
Part~(d) is a consequence of~\cite{GM}.
Part~(e) will follow from a combination of~\cite{GMX} and Part~(b).
Part~(f) will be proved using the Central Limit Theorem.

\begin{conj}\label{conj:ransplit}
  Let $\lambda_1,\ldots,\lambda_k,\lambda$ be positive numbers such that
  $\sum_{i=1}^k\lambda_i=\lambda$.
  Then, if $n\to\infty$ with $2\le m\le n$, and $1\le k=o(m)$,
  the average number of ways to partition the edges of a uniform
  random $(m,n,\lambda)$-semiregular bipartite graph into
  spanning semiregular subgraphs of density $\lambda_1,\ldots,\lambda_k$
  is asymptotically
   \begin{align*}
      &\frac{R'(m,n;1-\lambda,\lambda_1\ldots,\lambda_k)}
             {R'(m,n;1-\lambda,\lambda)} \\
       &{\kern 4em} =
       \frac{\displaystyle\binom{\lambda n}{\lambda_1 n,\ldots,\lambda_k n}^{\!m}
                                   \binom{\lambda m}{\lambda_1 m,\ldots,\lambda_k m}^{\!n}}
              {\displaystyle\binom{\lambda mn}{\lambda_1 mn,\ldots,\lambda_k mn}}
       \, (1-1/m)^{(k-1)(m-1)/2}.
   \end{align*}
\end{conj}

Note that Conjecture~\ref{conj:ransplit} holds if the formula in Theorem~\ref{thm:reg}
holds for $(m,n,\lambda)$ and Conjecture~\ref{conj:main} holds for
$(m,n,1-\lambda,\lambda_1,\ldots,\lambda_k)$.
Thus, many cases of Conjecture~\ref{conj:ransplit} follow from Theorem~\ref{thm:reg}
and Theorem~\ref{thm:main}.

\medskip

For integer $x\ge 0$, $(N)_x$ denotes the falling factorial
$N(N-1)\cdots(N-x+1)$.
\red{It will be convenient to have an approximation for
 $R'(m,n;\lambda_0,\ldots,\lambda_k)$
when $\sum_{i=1}^k\lambda_i$ is small.
\begin{lemma}\label{lem:Rapprox}
 Let $\lambda_0,\ldots,\lambda_k$ be positive numbers such that
$\sum_{i=0}^k\lambda_i=1$.
 Define $\lambda=\sum_{i=1}^k \lambda_i$ and assume $m\le n$ and
 $\lambda=O(m^{-1/4}n^{-1/4})$.  Then
  \begin{align}
        R'(m,n;\lambda_0,\ldots,\lambda_k) &=
       \prod_{i=1}^k \frac{ (\lambda_i mn)!}{(\lambda_i m)!^n\,(\lambda_i n)!^m} 
           \exp\( \varDelta_1(m,n,\lambda) + O(\lambda^4 mn)\) \label{eq:Rp4} \\
        &=
         \prod_{i=1}^k \frac{ (\lambda_i mn)!}{(\lambda_i m)!^n\,(\lambda_i n)!^m}
           \exp\( \varDelta_2(m,n,\lambda) + O(\lambda^3 mn)\), \label{eq:Rp3}
 \end{align}
 where
 \begin{align*}
 \varDelta_1(m,n,\lambda)&= -\dfrac12 k + \dfrac{k}{4m} - \dfrac12\lambda
           + \dfrac14 \lambda(2+\lambda)(m+n)
           - \dfrac16\lambda^2(3+\lambda)mn
           - \dfrac{1}{12}\lambda\(\dfrac mn+\dfrac nm\) \\
  \varDelta_2(m,n,\lambda) &= 
       -\dfrac12  k+ \dfrac12 \lambda(m+n) - \dfrac12\lambda^2 mn.
 \end{align*}
\end{lemma}
\begin{proof} 
If $N\to\infty$ and $\lambda=O(N^{-1/4})$, with $\lambda N$ integer,
\begin{align}
   (N)_{\lambda N} &= N^{\lambda N}
        \exp\biggl(\,\sum_{i=0}^{\lambda N-1}  \log \Bigl(1 - \dfrac iN\Bigr)\biggr) \notag\\  
     &= N^{\lambda N}  \exp\Bigl( -\sum_{r=1}^\infty \dfrac1r
        \sum_{i=0}^{\lambda N-1} \( i/N\)^r \Bigr) \notag \\
       &= N^{\lambda N} 
         \exp\Bigl( -\dfrac16\lambda^2(3+\lambda) N + \dfrac14 \lambda(2+\lambda)
        - \dfrac{\lambda}{12N} + O(\lambda^4 N) \Bigr). \label{eq:fff} 
\end{align}
Note that 
\[
     \binom{N}{\lambda_0,\ldots,\lambda_k} =
       \frac{(N)_{\lambda N}}{\prod_{i=1}^k (\lambda_i N)!}.
\]
Since $\lambda\ge\dfrac km$, $k/m^2=O(\lambda^4 mn)$, so we also have
\[
    (1 - 1/m)^{k(m-1)/2} = \exp\Bigl( -\dfrac12 k + \dfrac k{4m} + O(\lambda^4 mn)\Bigr).
\]
Applying this, together with~\eqref{eq:fff} for $N=m,n,mn$, gives~\eqref{eq:Rp4}.
Removing the terms that are $O(\lambda^3 mn)$ gives~\eqref{eq:Rp3}.
\end{proof}}

\section{A first case of a single dense factor}

In~\cite{silver}, the second author proved a generalized version of the
following lemma.

\begin{lemma}\label{lem:silver}
 Assume $m\le n$, $\lambda_h>0$, $\lambda_d\ge 0$
 and $\lambda_h(\lambda_d^2+\lambda_h^2)m^{-1}n^3=o(1)$.
 Let $D$ be any $(m,n,\lambda_d)$-semiregular graph.
 Then the number of $(m,n,\lambda_h)$-semiregular graphs edge-disjoint from~$D$ is
\[
   \frac{(\lambda_h mn)!}{(\lambda_h m)!^n(\lambda_h n)!^m}
   \exp\( -\dfrac12 (\lambda_h m-1)(\lambda_h n-1) 
     -\lambda_h\lambda_dmn + O(\lambda_h(\lambda_d+\lambda_h)^2m^{-1}n^3)\).
\]
\end{lemma}

Now we can prove Theorem~\ref{thm:main}(b).

\begin{thm}\label{thm:silver}
 Assume $m\le n$, \red{$k\ge 1$,} and that $\lambda_1,\ldots,\lambda_k$ are positive numbers
 with $\sum_{i=1}^k \lambda_i = o(m^{1/3}/n)$.
 Define $\lambda =\sum_{i=1}^k \lambda_i$ and $\lambda_0=1-\lambda$.
 Then
 \[
      R(m,n;\lambda_0,\ldots,\lambda_k) = R'(m,n;\lambda_0,\ldots,\lambda_k)
      \( 1 + O(\lambda^3 m^{-1}n^3)\).
 \]
\end{thm}
\begin{proof}
   \red{The case $k=1$ corresponds to $\lambda_d=0$ in Lemma~\ref{lem:silver},
   so assume $k\ge 2$.}
   For notational convenience, write
   the argument of the exponential in Lemma~\ref{lem:silver} as
   $A(\lambda_d,\lambda_h)+O(\delta(\lambda_d,\lambda_h))$.
   We can partition $K_{m,n}$ into spanning semiregular subgraphs of
   density $\lambda_0,\ldots,\lambda_k$ by first choosing an
   $(m,n,\lambda_1)$-semireg\-ular graph~$D_1$, then choosing an 
   $(m,n,\lambda_2)$-semiregular graph~$D_2$ disjoint from~$D_1$,
   then an $(m,n,\lambda_3)$-semiregular graph~$D_3$ disjoint from~$D_1\cup D_2$,
   and so on \red{until we have chosen disjoint $D_1,\ldots,D_k$.
   Since the density of $D_1\cup\cdots\cup D_j=\lambda_1+\cdots+\lambda_j$ 
   for each~$j$, we have}
   \begin{align*}
       R(m,n;&\lambda_0,\ldots,\lambda_k) =
         \exp\( \hat A + O(\hat\delta)\)
         \prod_{i=1}^k \frac{ (\lambda_i mn)!}{(\lambda_i m)!^n\,(\lambda_i n)!^m},
             \qquad\text{where} \\[-1ex]
        \hat A&= A(0,\lambda_1)+A(\lambda_1,\lambda_2)+A(\lambda_1+\lambda_2,\lambda_3)
             + \cdots + A\Bigl({\textstyle\sum_{i=1}^{k-1}}\lambda_i,\lambda_k\Bigr), \\
        \hat \delta&= \delta(0,\lambda_1)+\delta(\lambda_1,\lambda_2)+
        \delta(\lambda_1+\lambda_2,\lambda_3)
             + \cdots + \delta\Bigl({\textstyle\sum_{i=1}^{k-1}}\lambda_i,\lambda_k\Bigr) . 
   \end{align*}
   By routine induction, we find that
   \[
       \hat A = -\dfrac12  k + \dfrac12 \lambda(m+n) - \dfrac12 \lambda^2mn
   \]
   and $\hat\delta = O(\lambda^3 m^{-1}n^3)$.
   
   \red{Since $m\le n$ we have $mn=O(m^{-1}n^3)$, which completes the
   proof in conjunction with~\eqref{eq:Rp3},}
\end{proof}   

\section{A second case of one dense factor}\label{s:stepone}

The enumeration of sparse semiregular bipartite graphs was extended
to higher degrees by McKay and Wang~\cite{MW} and generalised by
Greenhill, McKay and Wang~\cite{GMW}.
However, unlike Lemma~\ref{lem:silver}, there was no allowance 
\red{in~\cite{GMW, MW}}
for a set of forbidden edges.  In order to use the more accurate enumeration
for our purposes, we consider the problem of forbidden edges in
a more general form that may be of independent interest. Instead of
considering a random semiregular graph, we consider an arbitrary
semiregular graph that is labelled at random. 

When we write of a semiregular bipartite graph on $(V_1,V_2)$ being
randomly (re)labelled, we mean that $V_1$ and $V_2$ are independently
permuted ($m!\,n!$ possibilities altogether of equal probability).
In this section we consider two semiregular bipartite graphs on $(V_1,V_2)$:
an $(m,n,\lambda_d)$-semiregular graph~$D$ and
an $(m,n,\lambda_h)$-semiregular graph~$H$.
If $H$ is randomly labelled, what is the probability that
it is edge-disjoint from~$D$?

Define $M = \lceil \lambda_d\lambda_h mn\log n\rceil$.
We will work under the following assumptions.
\begin{equation}\label{eq:assume}
   n\to\infty, \quad m\le n, \quad
   \lambda_d,\lambda_h>0, \quad \lambda_d^2\lambda_h^2 mn^2=o(1).
\end{equation}

\begin{lemma}\label{lem:assume}
 Under assumptions~\eqref{eq:assume}, we have
 %$m=\omega(n^{2/3})\to\infty$ and 
 \red{$m\to\infty$, $n=o(m^{3/2})$, and}
 $m^{-1}\le \lambda_d,\lambda_h = o(m^{1/2}n^{-1})$.
 In addition, $\lambda_d \lambda_h=o(m^{-1/2}n^{-1})$
 and $\log n\le M=o(m^{1/2}\log n)$.
\end{lemma}
\begin{proof}
  Since $D,H$ are not empty, $\lambda_d,\lambda_h\ge\frac1m$, corresponding
  to degree~1 in~$V_2$.  The rest is elementary.
\end{proof}

\begin{lemma}\label{lem:basic}
Let $D$ be an $(m, n, \lambda_d)$-semiregular bipartite graph and 
$H$ be an $(m, n, \lambda_h)$-semiregular bipartite graph
satisfying Assumptions~\eqref{eq:assume}.
Then, with probability $1- O(\lambda_d^2 \lambda_h^2 mn^2)$, a random labeling 
of $H$ does not have any path of length~2 in common with~$D$
and the number of edges in common with $D$ is less than~$M$.
\end{lemma}

\begin{proof}
Graphs $D$ and $H$ have $O(mn^2\lambda_d^2)$ and
$O(mn^2\lambda_h^2)$ paths of length~2 with the central vertex in~$V_1$,
\red{respectively}.
The probability that such a path in $H$ matches a \red{given} path in $D$ when
randomly labelled is $O(m^{-1}n^{-2})$.  So the expected number of
such coincidences is $O(\lambda_d^2 \lambda_h^2 mn^2)$.
Similarly, the expected number of coincidences between paths of
length~2 with central vertex in $V_2$ is $O(\lambda_d^2 \lambda_h^2 m^2n)$,
%which is no larger on account of the assumption $m\le n$.
\red{which is $O(\lambda_d^2 \lambda_h^2 mn^2)$ because $m\le n$
by assumption}.
Thus, with probability $1-O(\lambda_d^2 \lambda_h^2 m^2n)$, a
random labelling of $H$ has no paths of length~2 in common with~$D$.

Now consider a set $S$ of $M$ edges of $H$. In light of the preceding,
we can assume they are independent edges.  There are at most 
$\binom{\lambda_h mn}{M}$ choices of $S$, and at most
$\binom{\lambda_d mn}{M}$ choices of a set of $M$ edges of $D$ that
$S$ might map onto.
The  probability that $S$
maps onto a particular set of $M$ independent edges of~$D$ is
\[
    \frac{ M!\, (m-M)!\,(n-M)!}{m!\,n!} \le \frac{M!}{(m-M)^M (n-M)^M}.
\]

Since $M=o(m)$ and $M=o(n)$, $mn\le 2(m-M)(n-M)$
for sufficiently large~$n$.
Using $M!\ge (M/e)^M$, we find that the expected number of sets of
$M$ independent edges of $H$ that map onto edges of $D$ is at most
\[
    \biggl( \frac {\lambda_h\lambda_d\, m^2n^2 e}{M(m-M)(n-M)}\biggr)^{\!M}
    \le \biggl( \frac{2e}{\log n} \biggr)^{\!\log n} 
    = O(n^{-t}) \quad\text{for any $t>0$}.
\]
\red{By Lemma~\ref{lem:assume}, $\lambda_d^2\lambda_h^2mn^2\ge n^{-1}$,
so the probability that $D$ and $H$ have more than $M$ edges in common
is $O(\lambda_d^2\lambda_h^2mn^2)$. This completes the proof.}
%This probability is smaller than the probability of common paths of length~2,
%which completes the proof.
\end{proof}

Let $\calL(t)$ be the set  of all labelings of the vertices of $H$ with no common
paths of length 2 with $D$ and exactly $t$ edges in common with $D$.
Define $L(t)=|\calL(t)|$, so in particular
the number of labelings of the vertices of $H$ with no common edges with $D$ is $L(0)$. 
Let
    \[
        T = \sum_{t =0}^{M-1}  L(t).
    \]   
In the next step we will estimate the value of $T/L(0)$ by the switching method.
   
%************************************************************************

\bigskip

A \textit{forward switching} is a permutation $(a \,  e)(b \, f)$ of the vertices of $H$ such that
\begin{itemize}\itemsep=0pt
 \item $a,e\in V_1$ are distinct, and $b,f\in V_2$ are distinct.
 \item  $ab$ is a common edge of $D$ and $H$,
 \item  $ef$ is a non-edge of $D$, and 
 \item after the permutation, the edges common to $D$ and $H$ are the
         same except that $ab$ is no longer a common edge.
\end{itemize}

\noindent
A \textit{reverse switching} is a permutation $(a \,  e)(b \, f)$ of the vertices
      of $H$ such that
\begin{itemize}\itemsep=0pt
\item $a,e\in V_1$ are distinct, and $b,f\in V_2$ are distinct.
 \item  $ab$ is an edge of $D$ that is not an edge of $H$,
 \item  $ef$  is an edge of $H$ that is not an edge of $D$, and 
 \item after the permutation, the edges common to $D$ and $H$ are the
         same except that $ab$ \red{becomes} a common edge of $D$ and $H$.
\end{itemize}

\begin{lemma}\label{lem:Lrat}
Assume Conditions~\eqref{eq:assume}.
Then, uniformly for $1\le t\le M$,
\[
   \frac{L(t)}{L(t-1)} = \frac{(\lambda_d\, mn-t+1)(\lambda_h\, mn-t+1)}{tmn}
      \biggl(1 + O\Bigl(\frac tm + \frac{1}{m^{1/2}}\Bigr)\biggr).
\]
\end{lemma}

\begin{proof}
By using a forward switching, we will convert  a labeling $R \in \calL(t)$ to a
labeling $R'\in \calL(t-1)$. Without loss of generality,  we suppose that $R$ is the identity,
since our bounds will be independent of the structure of $H$ other
than its density.

There are $t$ choices for edge $ab$.
We will bound the choices of $e,f$ for a fixed choice of $ab$.
Graph $D$ has $(1-\lambda_d)mn$ non-edges $ef$, including
at most $m+n=O(n)$ for which $e=a$ or $f=b$.
In addition, we must ensure that no common edges are created
\red{(implying that there remain no common paths of length~2)},
and none other than $ab$ are destroyed.

Since  $D$ and $H$  have no paths of  length 2 in common, there are no other
common edges of $H$ and $D$ incident to $a$ or $b$.
Therefore the only way to destroy a common edge other than $ab$ is
if $e$ or $f$ are incident to a common edge. 
This eliminates at most $t(n+m)=O(tn)$ pairs $e,f$.

Creation of a new common edge can only occur if there is a path
\red{of length~2} from~$a$
to~$e$, or from~$b$ to~$f$, consisting of one edge from $H$ and one from~$D$.
This eliminates at most $4\lambda_d\lambda_h mn$ choices of $ef$.

Using Lemma~\ref{lem:assume} we find that
 the number of forward switchings is
\begin{align*}
  W_F &= t\( (1-\lambda_d)mn - O(n+tn+\lambda_d\lambda_h mn)\) \\
     &= t \,mn \biggl(1 + O\Bigl(\frac tm + \frac{m^{1/2}}n\Bigr)\biggr).
\end{align*}

A reverse switching converts a labeling $R' \in \calL(t-1)$ to a labeling $R$ in $\calL(t)$.
Again, without loss of generality, we suppose that $R'$ is the identity. There are
$\lambda_d\, mn-t+1$ choices for an edge $ab$ in~$D$ and $\lambda_h mn-t+1$ choices
\red{for an} edge~$ef$ in~$H$, in each case avoiding the $t-1$ common edges.
From these we must subtract any choices that create or destroy a common edge
except the new common edge $ab$ we intend to create,
\red{as well as those with $a=e$ or $b=f$}.

There are at most $\lambda_d\lambda_h mn^2$ choices of $a,b,e,f$ such that
$a=e$ and at most that number (since $m\le n$) such that $b=f$.

To destroy a common edge, at least one of $a,b,e,f$ must be the
endpoint of a common edge.  The number of such choices is bounded by
$4(t-1)\lambda_d\lambda_hmn^2$.

To create a new common edge other than $ab$, there must be a path \red{of length~2}
from~$a$ to~$e$, or from~$b$ to~$f$, consisting of one edge from $H$ and
one from~$D$.  This eliminates at most $\lambda_d^2\lambda_h^2m^2n^3$ cases.

Using Lemma~\ref{lem:assume} we find that
 the number of \red{reverse} switchings is
\begin{align*}
  W_R &= (\lambda_d mn-t+1)(\lambda_h mn-t+1) 
     + O\( t\lambda_d\lambda_h mn^2 + \lambda_d^2\lambda_h^2 m^2n^3\) \\
     &= (\lambda_d mn-t+1)(\lambda_h mn-t+1) 
           \biggl( 1 + O\Bigl(\frac tm + \frac{1}{m^{1/2}} \Bigr)\biggr).
\end{align*}

The lemma now follows from the ratio $W_R/W_F$, using $m\le n$.
\end{proof}
%%%%%%%%%%%%%%%%%%%%%%%%%%%%%%%%%%%%%%%%%%%%%%%%%

We will need the following summation lemma from \cite[Cor.~4.5]{GMW}.

\begin{lemma}\label{sumlemma}
Let $Z \ge 2$ be an integer and, for $1\le i \le Z$, let real numbers $A(i)$, $B(i)$
be given such that $A(i) \ge 0$ and $1 - (i - 1)B(i) \ge 0$. Define
 $A_1 = \min _{i=1}^{Z} A(i)$, $A_2 =\max_{i=1}^{Z} A(i), 
 C_1 = \min_{i=1}^{Z} A(i)B(i)$ and $C_2 = \max_{i=1}^{Z} A(i)B(i)$.
 Suppose that there exists
$\hat{c}$ with $0 < \hat{c} < \frac13$
 such that $\max\{A/Z, |C|\} \le \hat{c}$
 for all $A \in [A_1,A_2]$,
$C \in [C_1,C_2]$. Define $n_0, \ldots , n_Z$ by $n_0 = 1$ and
\[ n_i/ n_{i-1}=\dfrac1i A(i) (1 - (i -1)B(i)) \]
for $1 \le i \le Z$, with the following interpretation: if $A(i) = 0$ or $1 - (i - 1)B(i) = 0$, then
$n_j = 0$ for $ i \le j \le Z$. Then 
\[
   \Sigma_{1} \le \sum_{i=0}^{Z}n_i \le \Sigma_2,
\]
where
\[
  \Sigma_1 = \exp\(A_1 - \dfrac12 A_1C_2)-(2e\hat{c}\)^Z
\]
and 
\[
    \Sigma_2=\exp\(A_2-\dfrac12A_2C_1+\dfrac12 A_2C_1^2\)+ (2e\hat{c} )^Z. 
    \hbox to0pt{\qquad\qquad\qed\hss}
\]
\end{lemma}

\begin{lemma}\label{lem:tsum}
Under Assumptions~\eqref{eq:assume} we have
\[
     \frac{T}{L(0)} = \exp\( \lambda_d\lambda_h mn + O(\lambda_d\lambda_h m^{1/2}n) \).
\]
\end{lemma}

\begin{proof}
Lemma~\ref{lem:assume} and the definition of $M$
allow us to write Lemma~\ref{lem:Lrat} as
\[
   \frac{L(t)}{L(t-1)} = \dfrac 1t A(t) \(1-(t-1)B(t)\),
\]
where $A(t)=\lambda_d\lambda_h mn(1+O(m^{-1/2}))$
and $B(t)=O(m^{-1})$.
In each case, the $O(\,)$ expression is a function of $t$
but uniform over $1\le t\le M$.

Clearly $A(t)>0$, and using Lemma~\ref{lem:assume}, we can
check that $1-(t-1)B(t)>0$ for $1\le t\le M$.
Define $A_1$, $A_2$, $C_1$ and $C_2$ as in Lemma~\ref{sumlemma}. 
This gives
\begin{align*}
A_1,A_2&=  \lambda_d\lambda_h mn + O(\lambda_d\lambda_hm^{1/2}n), \\
C_1,C_2&= O(\lambda_d\lambda_h n) = o(m^{-1/2}).
\end{align*}

The condition $\max\{A/M, |C|\} \le \hat{c}$ of Lemma~\ref{sumlemma} is satisfied
with $\hat{c}$ any sufficiently small constant.
Using Lemma~\ref{lem:assume}, we also have
$A_1C_2, A_2C_1= O(\lambda_d^2 \lambda_h^2 mn^2)$,
$A_2C_1^2=O(\lambda_d^2 \lambda_h^2 m^{1/2}n^2)$
and $(2e\hat{c})^M=O(n^{-1})$ if $\hat{c}$ is
small enough.
Since $\lambda_d^2 \lambda_h^2 mn^2=o(1)$ by assumption,
$\lambda_d^2 \lambda_h^2 mn^2=o(\lambda_d \lambda_h m^{1/2}n)$.
The lemma now follows from Lemma~\ref{sumlemma}.
\end{proof}

\begin{lemma}\label{lem:prob}
Let $D$ be an $(m,n, \lambda_d)$-semiregular bipartite graph and
$H$ be an $(m,n, \lambda_h)$-semiregular bipartite graph such that
Assumptions~\eqref{eq:assume} hold.
Then the probability that
a random labelling of $H$ is edge-disjoint from $D$ is
\[
   \exp\( -\lambda_d\lambda_h mn + O(\lambda_d\lambda_h m^{1/2}n) \).
\]
\end{lemma}
\begin{proof}
The probability that there are no common paths of length two and less
than $M$ edges in common is $1-O(\lambda_d^2\lambda_h^2 mn^2)$
by Lemma~\ref{lem:basic}.
Subject to those conditions, the probability of no common edge is
\[
   \exp\( -\lambda_d\lambda_h mn + O(\lambda_d\lambda_h m^{1/2}n) \)
\]
by Lemma~\ref{lem:tsum}.  Multiplying these probabilities together
gives the theorem, since\\ $\lambda_d\lambda_h m^{1/2}n\to 0$ by~\eqref{eq:assume}.
\end{proof}

\begin{thm}\label{thm:disj}
For $i=1, \ldots, k$, let $D_i$ be an $(m,n,\lambda_i)$-semiregular  bipartite graph.
Assume $m \le n$ and $m^{1/2}n\sum_{1\le i<j\le k} \lambda_i\lambda_j=o(1)$.
Then the probability that $D_1,D_2,\ldots,D_k$ are edge-disjoint after
labelling at random is
\[
\exp\biggl(-mn \sum_{1 \le i < j \le k} \lambda_i \lambda_j
  +O\Bigl(m^{1/2}n \sum_{1 \le i < j \le k} \lambda_i \lambda_j\Bigr)\biggr).
\]
\end{thm}
\begin{proof}
This is an immediate consequence of Lemma~\ref{lem:prob} applied
iteratively.
\end{proof}

Now we are ready to prove Theorem~\ref{thm:main}(c).
First we state the enumeration theorem from~\cite{MW}, extended with
the help of Lemma~\ref{lem:prob}.
Note that, despite this theorem and Lemma~\ref{lem:silver} having considerable
overlap, neither of them implies the other.

\begin{thm}\label{thm:MW}
 Assume $m\le n$, $\lambda_h>0$, $\lambda_d\ge 0$
 and $\lambda_h\lambda_d\, m^{1/2}n=o(1)$.
 Let $D$ be any $(m,n,\lambda_d)$-semiregular graph.
 Then the number of $(m,n,\lambda_h)$-semiregular graphs edge-disjoint from~$D$ is
\begin{align*}
   \frac{(\lambda_h mn)!}{(\lambda_h m)!^n(\lambda_h n)!^m}
   &\exp\( -\dfrac12 (\lambda_h m-1)(\lambda_h n-1) 
     -\dfrac16\lambda_h^3mn \\[-1ex]
    &{\qquad\quad} -\lambda_h\lambda_d\,mn + O(\lambda_h\lambda_d\, m^{1/2}n)\).
\end{align*}
\end{thm}

\begin{thm}\label{thm:coloured}
  Let $\lambda_0,\lambda_1,\ldots,\lambda_k>0$ be such that $k\ge 2$ and
  \red{$\sum_{i=0}^k \lambda_i=1$}.  
  Assume that $m \le n$ and $m^{1/2}n\sum_{1\le i<j\le k} \lambda_i\lambda_j=o(1)$.
 Then
 \[
  R(m,n;\lambda_0,\ldots,\lambda_k) =
    R'(m,n;\lambda_0,\ldots,\lambda_k)
     \biggl(1 + O\Bigl(m^{1/2}n\sum_{1\le i<j\le k} \lambda_i\lambda_j \Bigr)\biggr),
\]
where
 $R'(m,n;\lambda_0,\ldots,\lambda_k)$ is defined in Conjecture~\ref{conj:main}.
\end{thm}
\begin{proof}
The proof follows the same line as Theorem~\ref{thm:silver}.
Define $\lambda=\sum_{i=1}^k\lambda_i$ and
$\varLambda=\sum_{1\le i<j\le k} \lambda_i\lambda_j$. Write
   the argument of the exponential in Theorem~\ref{thm:MW} as
   $A'(\lambda_d,\lambda_h)+O(\delta'(\lambda_d,\lambda_h))$.
   Then, as before,
   \begin{align*}
       R(m,n;&\lambda_0,\ldots,\lambda_k) =
         \exp\( \check A + O(\check\delta)\)
         \prod_{i=1}^k \frac{ (\lambda_i mn)!}{(\lambda_i m)!^n\,(\lambda_i n)!^m},
             \qquad\text{where} \\[-1ex]
        \check A&= A'(0,\lambda_1)+A'(\lambda_1,\lambda_2)+A'(\lambda_1+\lambda_2,\lambda_3)
             + \cdots + A'\Bigl({\textstyle\sum_{i=1}^{k-1}}\lambda_i,\lambda_k\Bigr), \\
        \check \delta&= \delta'(0,\lambda_1)+\delta'(\lambda_1,\lambda_2)+
        \delta'(\lambda_1+\lambda_2,\lambda_3)
             + \cdots + \delta'\Bigl({\textstyle\sum_{i=1}^{k-1}}\lambda_i,\lambda_k\Bigr) . 
   \end{align*}
   By routine induction, we find that
   \[
       \check A = -\dfrac12  k + \dfrac12 \lambda(m+n) 
           -\dfrac16 mn\sum_{i=1}^k\lambda_i^3- \dfrac12 \lambda^2mn
       \quad\text{and}\quad \check\delta = m^{1/2}n\varLambda.
   \]
   
   \red{Since $\lambda_i\ge\frac1m$ for $1\le i\le k$, we have
   $\varLambda=\frac12\sum_{i=1}^k\lambda_i\sum_{j\ne i}\lambda_j
   \ge \frac{k-1}{2m}\lambda\ge \frac{\lambda}{2m}$.
   Therefore, the condition $m^{1/2}n\varLambda=o(1)$ implies that
   $m^{-1/2}n\lambda=o(1)$ and $\check\delta=\Omega(m^{-1/2}n\lambda)$.
   Since $\lambda^4 mn = (m^{-1/2}n\lambda)^4 m^3n^{-3}$ and $m\le n$,
   we also have $\lambda^4 mn=O(\check\delta)=o(1)$.
   This allows us to apply~\eqref{eq:Rp4} to estimate $R'(m,n;\lambda_0,\ldots,\lambda_k)$.
   We also have $km^{-1}\le\lambda=O(\check\delta)m^{1/2}n^{-1}=O(\check\delta)$,
   $\lambda^2m\le \lambda^2 n=O(\check\delta^2)mn^{-1}=O(\check\delta)$, and
   $\lambda mn^{-1}\le\lambda nm^{-1}=O(\check\delta^2)m^{-1/2}=O(\check\delta)$,
   which eliminates many terms of~\eqref{eq:Rp4}.
   What remains gives
   \[
     R(m,n;\lambda_0,\ldots,\lambda_k) = R'(m,n;\lambda_0,\ldots,\lambda_k)
     \exp\biggl( \dfrac16 \lambda^3 mn - \dfrac16 mn\sum_{i=1}^k\lambda_i^3
        + O(\check\delta)\biggr).
   \]
   Finally, we have
   \[
     mn\biggl( \lambda^3 - \sum_{i=1}^k\lambda_i^3\biggr)
       = O(mn)\sum_{\ell=1}^k\sum_{i<j}\lambda_i\lambda_j\lambda_\ell
       = O(\lambda mn\varLambda) = O(\check\delta^2)m^{1/2}n^{-1/2} = O(\check\delta),
    \]
   which completes the proof.
   }
\end{proof}

\begin{figure}[ht!]
   \[
      \includegraphics[scale=0.4]{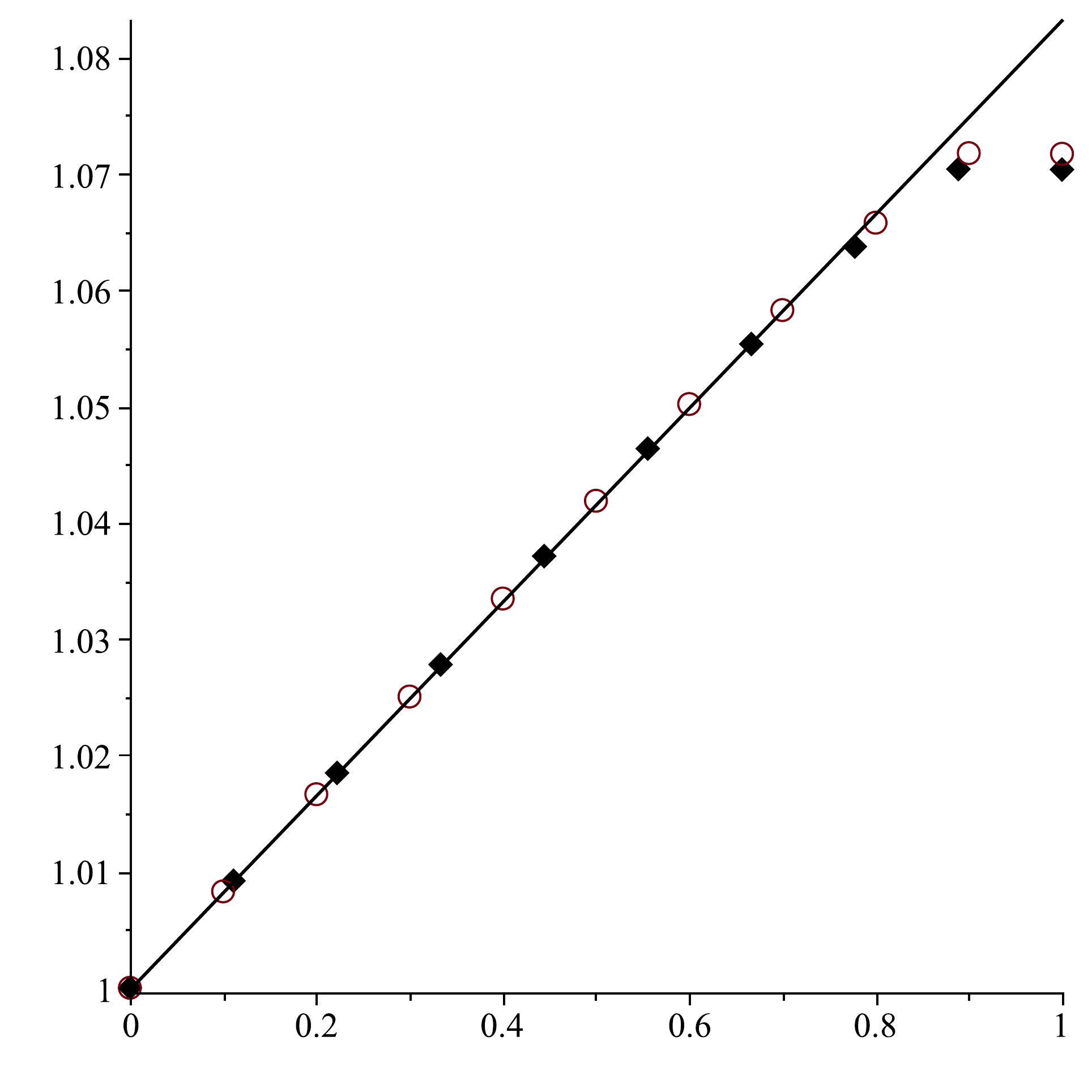}
  \]
      \vspace*{-6ex}
      \caption{Latin rectangles.
   $F(n,k)/R'(n,n;1-\frac kn,\frac1n,\ldots,\frac1n)$ for $n=10$ (diamonds)
   and $n=11$ (circles). The horizontal scale is $x=k/(n-1)$
   and the line is $1+x/12$.\label{fig}}
\end{figure}

\section{The case of Latin rectangles}

The case $m=n, \lambda_1,\ldots,\lambda_k=\frac{1}{n}$
corresponds to choosing an ordered sequence of $k$ perfect matchings
in $K_{n,n}$, which is a $k\times n$ Latin rectangle.

Let $F(n,k)$ be the number $k\times n$ Latin rectangles.
Note that $F(n,n-1)=F(n,n)$; we will use only $F(n,n-1)$.
In~\cite{GM}, Godsil and McKay found the asymptotic value of $F(n,k)$
for $k=o(n^{6/7})$ and conjectured that the same formula holds for any
$k=O(n^{1-\delta})$ with $\delta>0$, namely that 
\[
  F(n,k)
  \sim (n!)^k \biggl( \frac{n!}{(n-k)!\,n^k}\biggr)^{\!n}
     \biggl(1 - \frac{k}{n}\biggr)^{\!-n/2} e^{-k/2}.
\]
The reader can check that this expression is equal to
%$R'(n,n;k\times\frac1n)$ asymptotically when $k=o(n)$.
%Here we have used $k\times\frac1n$ to represent a sequence
%of $k$ terms with each term equal to~$\frac1n$.
\red{$R'(n,n; 1-\frac kn,\frac1n,\ldots,\frac1n)$ asymptotically when
$k=o(n)$, where here and in the following $\frac1n$ occurs $k$
times as an argument of~$R'$.}

To illustrate what happens when $k$ is even larger, Figure~\ref{fig}
shows the ratio
$F(n,k)/\allowbreak \red{R'(n,n;1-\frac kn,\frac1n,\ldots,\frac1n)}$
for $n=10,11$, as a function
of $k/(n-1)$~\cite{LS11}.                                                                                                                                                                                                                                                                                                                                                                                                                                                                                                                                                                                                                                                                                                                                                                                                                                                                                                                                                                                                                                                                                                                                                                                                                                                                                                                                                                                                                                                                                                                                                                                                                                                                                                                                                                                                                                                                                                                                                                                                                                                                                                                                                                                                                                                                                                          
Experiment suggests that $F(n,x(n-1))/\red{R'(n,n;1-\frac kn,\frac1n,\ldots,\frac1n)}$
converges to a continuous function $f(x)$ as $n\to\infty$ with
$x$ fixed. Conjecture~\ref{conj:main} in this case corresponds
to $f(0)=1$.

\section{Two factors of high degree}

In this section, we consider the case where $\lambda_0$
and $\lambda_1$ are approximately constant.
We will use a constant $\eps>0$ that must be sufficiently small.
Suppose the following conditions hold.
\begin{equation}\label{eq:dense}
\begin{aligned}
 (1-2\lambda_1)^2\biggl(1+\frac{5m}{6n}+\frac{5n}{6m}\biggr)
                     &  \le(4-\eps)\lambda_1(1-\lambda_1)\log n \\
m\le n &= o\(\lambda_1(1-\lambda_1) m^{1+\eps}\).
\end{aligned}
\end{equation}

Let $D$ be an $(m,n,\hat\lambda)$-semiregular graph where
$0<\hat\lambda\le n^{-1+\eps}$.
Greenhill and McKay~\cite{GMX} proved that the number of
$(m,n,\lambda_1)$-semiregular graphs is given by Theorem~\ref{thm:reg},
\red{and the} probability that a uniformly random
$(m,n,\lambda_1)$-semiregular graph is edge-disjoint from~$D$
is asymptotically
\begin{equation}\label{eq:l01}
    P(m,n,\lambda_1,\hat \lambda) = (1-\lambda_1)^{\hat\lambda mn}  \exp\biggl(
      -\frac{\lambda_1\hat\lambda(\hat\lambda mn-m-n)}
              {2(1-\lambda_1)} \biggr).
\end{equation}

\begin{thm}\label{thm:ranx}
Suppose conditions~\eqref{eq:dense} hold and suppose
$\hat\lambda=\lambda_2+\cdots+\lambda_k=O(n^{-1+\eps})$
with sufficiently small constant $\eps>0$.
Then, as $n\to\infty$,
\[
   R(m,n; \lambda_0,\ldots,\lambda_k) \sim R'(m,n; \lambda_0,\ldots,\lambda_k).
\]
\end{thm}
\begin{proof}
\red{Our approach is to first construct an arbitrary graph $D$ of density
$\hat\lambda$ partitioned into factors of density $\lambda_2,\ldots,\lambda_k$.
Then Theorem~\ref{thm:reg} together with~\eqref{eq:l01} tells us the number
of graphs $G$ of density $\lambda_1$ disjoint from~$D$. The complement
of $G\cup D$ is then the factor of density~$\lambda_0$.}

\red{The second part of Condition~\eqref{eq:dense} implies that
$n^{-1+\eps}\le m^{1/3}/n$ for sufficiently small $\eps>0$.
Therefore we can apply Theorem~\ref{thm:silver} to deduce that
the possibilities for~$D$ are}
\[
R(m,n; \lambda_0+\lambda_1,\lambda_2,\ldots,\lambda_k)\sim
    R'(m,n; \lambda_0+\lambda_1,\lambda_2,\ldots,\lambda_k).
\]
\red{In addition, $R(m,n;1-\lambda_1,\lambda_1)\sim 
R'(m,n;1-\lambda_1,\lambda_1)$ by Theorem~\ref{thm:reg} part~2.}
Therefore, 
\[
    R(m,n; \lambda_0,\ldots,\lambda_k)\sim P(m,n,\lambda_1,\hat\lambda)
             R'(m,n;1-\lambda_1,\lambda_1)
             R'(m,n;\lambda_0+\lambda_1,\lambda_2,\ldots,\lambda_k).
\]
Then we have
\[
    \frac{R(m,n; \lambda_0,\ldots,\lambda_k)}{R'(m,n; \lambda_0,\ldots,\lambda_k)}
    \sim
    P(m,n,\lambda_1,\hat\lambda)
     \frac{
              \Bigl( \frac{m!\, ((1-\lambda_1-\hat\lambda)m)!}
                      {((1-\lambda_1)m)!\, ((1-\hat\lambda)m)!} \Bigr)^{\!n}
              \Bigl( \frac{n!\, ((1-\lambda_1-\hat\lambda)n)!}
                      {((1-\lambda_1)n)!\, ((1-\hat\lambda)n)!} \Bigr)^{\!m}
            }  
            {
               \Bigl( \frac{(mn)!\, ((1-\lambda_1-\hat\lambda)mn)}
                      {((1-\lambda_1)mn)!\, ((1-\hat\lambda)mn)!} \Bigr)
            }.
\]
Define $g(N)$ by $N! =\sqrt{2\pi}\, N^{N+1/2} e^{-N+g(N)}$ and
\[
\bar g(N)= g(N)+g((1{-}\hat\lambda{-}\lambda_1)N)-g((1{-}\hat\lambda)N)-g((1{-}\lambda_1)N.
\]
Then, \red{using~\eqref{eq:l01},}
\begin{align*}
  &\frac{R(m,n; \lambda_0,\ldots,\lambda_k)}{R'(m,n; \lambda_0,\ldots,\lambda_k)} \\
  &{\qquad}\sim
  \( 1 - \hat\lambda\)^{-(1-\hat\lambda)mn-(m+n-1)/2}
  \biggl(1 - \frac{\hat\lambda}{1-\lambda_1}\biggr)^{\!(1-\hat\lambda-\lambda_1)mn+(m+n-1)/2} \\
  &{\qquad\qquad\qquad}\times \exp\biggl( -\frac{\lambda_1\hat\lambda(\hat\lambda mn-m-n)}
              {2(1-\lambda_1)} +n\bar g(m)+m\bar g(n)-\bar g(mn) \biggr).
\end{align*}
Now we can apply the estimates $1-x=e^{-x-x^2/2+O(x^3)}$ and
$g(N)=\frac1{12N} + O(N^{-3})$ to show that the above quantity is $e^{o(1)}$.
This completes the proof.
\end{proof}

\section{The highly oblong case}

Suppose $2\le m=O(1)$ and $1\le k\le m-1$.  In that case we can prove
Conjecture~\ref{conj:main} by application of the Central Limit Theorem,
without requiring the condition $k=o(m)$.

We will consider the partition of $K_{m,n}$ into $k+1$ semiregular graphs
as an edge-colouring with colours $0,1,\ldots,k$, where colour $c$
gives a semiregular subgraph of density $\lambda_c$ for $0\le c\le k$.
Label $V_1$ as $u_1,\ldots,u_m$ and consider one vertex $v\in V_2$.
For $0\le c\le k$, $v$ must be adjacent to $\lambda_c m$ vertices of~$V_1$
by \red{edges} of colour~$c$;
let us make the choice uniformly at random from the
\[  \binom{m}{\lambda_0 m,\ldots,\lambda_k m} \] 
possibilities.  Define the random variable $X_{i,c}$ to be the indicator
of the event ``$v$~is joined to vertex~$u_i$ by colour~$c$''.
Let $\boldsymbol{X}$ be the $(m-1)k$-dimensional random vector
$(X_{i,c})_{1\le i\le m-1, 1\le c\le k}$.  Note that we are omitting
$i=m$ and $c=0$ since those indicators can be determined from
the others (this avoids degeneracy in the following).
\red{Let $\boldsymbol{X}^{(n)}$ denote the sum of $n$ independent copies
of $\boldsymbol{X}$, corresponding to a copy of $\boldsymbol{X}$ for
each vertex in~$V_2$.
For each vertex in $V_2$ to have the correct number $\lambda_c m$
of incident edges of each colour~$c$, $\boldsymbol{X}^{(n)}$ must equal its mean.
That is,
\[
   R(m,n;\lambda_0,\ldots,\lambda_k) =
      \binom{m}{\lambda_0 m,\ldots,\lambda_k m}^{\! n}  \,
       \mathrm{Prob}( \boldsymbol{X}^{(n)}=\E\boldsymbol{X}^{(n)}  ).
\]
We have $\E X_{i,c}=\lambda_c$ for every $i,c$ and the
following expectations of products.
\[
    \E(X_{i,c} \,X_{i',c'}) =
       \begin{cases}
           \lambda_c, & \text{if $i=i', c=c'$}; \\
           0, & \text{if $i=i', c\ne c'$}; \\
            \frac{\lambda_c(\lambda_c m-1)}{m-1},  & \text{if $i\ne i', c=c'$}; \\
           \frac{m}{m-1}\lambda_c\lambda_{c'}, & \text{if $i\ne i', c\ne c'$}.
     \end{cases}
\]
Using $\mathrm{Cov}(X_{i,c}, X_{i',c'}) 
  =\E(X_{i,c}\, X_{i',c'}) -\E X_{i,c}\,\E X_{i',c'}$ this gives}
\[
    \mathrm{Cov}(X_{i,c}, X_{i',c'}) =
       \begin{cases}
           \lambda_c(1-\lambda_c), & \text{if $i=i', c=c'$}; \\
           -\lambda_c\lambda_{c'}, & \text{if $i=i', c\ne c'$}; \\
           -\frac{\lambda_c(1-\lambda_c)}{m-1}, & \text{if $i\ne i', c=c'$}; \\
           \frac{\lambda_c\lambda_{c'}}{m-1}, & \text{if $i\ne i', c\ne c'$}.
     \end{cases}
\]
Let $\varSigma$ be the covariance matrix of $\boldsymbol{X}$,
labelled in the order $(1,1),(1,2),\ldots,(m-1,k)$.

By~\cite[Thm.~1]{CLT}, $\boldsymbol{X}^{(n)}$
satisfies a local Central Limit Theorem as $n\to\infty$.  In particular,
\red{since the covariance matrix of $\boldsymbol{X}^{(n)}$ is $n\varSigma$,}
\[
     \mathrm{Prob}( \boldsymbol{X}^{(n)}=\E\boldsymbol{X}^{(n)}  )
     \sim \frac{1}{(2\pi)^{k(m-1)/2} n^{k(m-1)/2} \abs{\varSigma}^{1/2}}.
\]
To find the determinant $\abs\varSigma$,
it helps to notice that $\varSigma$ is a tensor product
$B\otimes C$.  Here $B$ is a $k\times k$ matrix with
$B_{ii}=\lambda_i(1-\lambda_i)$ for all~$i$ and
$B_{ij}=-\lambda_i\lambda_j$ for $i\ne j$; while $C$ is an
$(m-1)\times(m-1)$ matrix with 1 on the diagonal and 
$-\frac1{m-1}$ off the diagonal.  Both $B$ and $C$ are rank-1 modifications
of diagonal matrices, and the matrix determinant lemma gives
$\abs{B}=\prod_{i=0}^k\lambda_i$ and $\abs{C}=m^{m-2}/(m-1)^{m-1}$.
Therefore,
\[
    \abs{\varSigma} = \abs{B}^{m-1}\abs{C}^k
%       = \biggl( \frac{m^{m-2}}{(m-1)^{m-1}}\biggr)^{\!k}
%         \biggl( \,\prod_{i=0}^k \lambda_i\biggr)^{\!m-1}
%       {\kern-1em}
       = m^{-k} (1-1/m)^{-k(m-1)} \biggl( \,\prod_{i=0}^k \lambda_i\biggr)^{\!m-1}.
\]
To complete the proof of Theorem~\ref{thm:main}(e),
apply \red{$N! =  \sqrt{2\pi}\,N^{N+1/2}e^{-N+O(1/N)}$} to find
\[
        \frac{\displaystyle\binom{n}{\lambda_0 n,\ldots,\lambda_k n}^{\!m}}
              {\displaystyle\binom{mn}{\lambda_0 mn,\ldots,\lambda_k mn}}
        \sim (2\pi n)^{-k(m-1)/2} m^{k/2} \biggl(\,\prod_{i=0}^k \lambda_i\biggr)^{\!-(m-1)/2}.
\]

\section{Concluding remarks}\label{s:conclusion}

We have proposed an asymptotic formula for the number of ways
to partition a complete bipartite graph into spanning semiregular
regular subgraphs and proved it in several cases.
The analytic method described in~\cite{mother} will be sufficient
to test the conjecture when there are several factors of high density.
This will be the topic of a future paper.
A similar investigation for regular graphs that are not necessarily
bipartite  appears in~\cite{RegPart}.

The authors declare that they have no conflict of interest.

\nicebreak
%*************************************************************************** 

\end{document}